\documentclass{amsart}
\usepackage[utf8]{inputenc}
\usepackage{amsmath, amssymb, amsfonts}
\usepackage{xcolor}

\newtheorem{theorem}{Theorem}[section]
\newtheorem{proposition}[theorem]{Proposition}
\newtheorem{conjecture}[theorem]{Conjecture}
\newtheorem{definition}[theorem]{Definition}
\newtheorem{lemma}[theorem]{Lemma}

\newtheorem{remark}[theorem]{Remark}
\newtheorem{problem}[theorem]{Problem}
\newtheorem{example}[theorem]{Example}

\newcommand{\KX}{K\langle X\rangle}

\title[Images of multilinear polynomials on $UT_n(K)$]{Images of multilinear polynomials on $n\times n$ upper triangular matrices over infinite fields}

\author[I. G. Gargate]{Ivan Gonzales Gargate}
\address{Universidade Tecnol\'ogica Federal do Paran\'a, campus Pato Branco}
\email{ivangargate@utfpr.edu.br}

\author[T. C. de Mello]{Thiago Castilho de Mello}
\address{Universidade Federal de S\~ao Paulo, Instituto de Ci\^encia e Tecnologia}
\email{tcmello@unifesp.br}

\date{December 2020}
\keywords{images of polynomials, multilinear polynomials, upper triangular matrices, Lvov-Kaplansky conjecture.}
\subjclass{16S50, 16R10, 15A54}

\begin{document}

\maketitle 

\begin{abstract}
  In this paper we prove that the image of multilinear polynomials evaluated on the algebra $UT_n(K)$ of $n\times n$ upper triangular matrices over an infinite field $K$ equals $J^r$, a power of its Jacobson ideal $J=J(UT_n(K))$. In particular, this shows that the analogue of the Lvov-Kaplansky conjecture for $UT_n(K)$ is true, solving a conjecture of Fagundes and de Mello. To prove that fact, we introduce the notion of \emph{commutator-degree} of a polynomial and characterize the multilinear polynomials of commutator-degree $r$ in terms of its coefficients. It turns out that the image of a multilinear polynomial $f$ on $UT_n(K)$ is $J^r$ if and only if $f$ has commutator degree $r$.
\end{abstract}

\section{Introduction}

Let $K$ be an infinite field and let $M_n(K)$ denote the algebra of $n\times n$ matrices over $K$. A famous problem known as Lvov-Kaplansky conjecture asserts: the image of a multilinear polynomial (in noncommutative variables) on $M_n(K)$ is a vector space. It is well-known that this is equivalent to prove that the image of a multilinear polynomial is one of the following: $\{0\}$, $K$ (viewed as the set of scalar matrices), $sl_n(K)$ (the set of traceless matrices) or $M_n(K)$.

Although proving that some subset is a vector space seems to be in a first look a simple problem, a solution for the Lvov-Kaplansky conjecture is known only for $n=2$ \cite{Kanel2, Malev}. The case $n=3$ has interesting progress, but not a solution \cite{Kanel3}. 
This conjecture motivated other studies related to images of polynomials. For instance, papers on images of Lie, and Jordan polynomials, and also for other algebras have been published since then. For the most recent results on images of polynomials, we recommend the survey \cite{Survey}.

The particular case of $n\times n$ upper triangular matrices over a field $K$, $UT_n(K)$, was first studied by the second named author and Fagundes in \cite{FagundesdeMello}. The authors proved that if $K$ is infinite, the image of a multilinear polynomial of degree up to four on $UT_n(K)$ is  $\{0\}$, $J$ or $J^2$, where $J=J(UT_n(K))$ is the Jacobson radical of $UT_n(K)$, which consists of the strictly upper triangular matrices. In the same paper, the authors conjecture that the image of a multilinear polynomial on $UT_n(K)$ is a vector space. After proving that the linear span of the image of an arbitrary multilinear polynomial on $UT_n(K)$ equals $J^r$ for some $r$, they restate the above conjecture as:

\begin{conjecture}\label{LKUTn}
    Let $K$ be an infinite field and let $f\in \KX$ be a multilinear polynomial. Then the image of $f$ on $UT_n(K)$ equals $J^r$ for some $r$. 
\end{conjecture}

In \cite{deMello}, the second named author verified this conjecture for $n=3$ and polynomials of arbitrary degree but up to now a complete solution for this conjecture is not known.

The paper \cite{FagundesdeMello} also motivated further research on images of polynomials on $UT_n(K)$. For instance, the images of multilinear polynomials on $J$ (and also on $J^k$, for $k\geq 1$) were completely described by Fagundes in \cite{Fagundes}. In \cite{WangHomogeneous} the authors give a complete description of the images of completely homogeneous polynomials on $UT_2$ and in \cite{WangArbitrary} they give a complete description of the image of polynomials with zero constant term on $UT_2$ when $K$ is an algebraically closed field. Some results were also obtained in the nonassociative setting. Let $K(n,*)$ denote the Lie algebra of $n\times n$ skew-symmetric elements of $UT_n(K)$ with respect to an involution of $UT_n(K)$. In \cite{FrancaUrure1} the authors give a complete description of the image of multihomogeneous Lie polynomials evaluated on $K(3,*)$ and $K(4,*)$ and of all multilinear Lie polynomials whose image evaluated on $K(m,*)$ is the set of strictly upper triangular skew-symmetric matrices (with respect to transpose-like and symplectic-like involutions). In \cite{FrancaUrure2} the authors proved that the image of any multilinear Jordan polynomial with at most three variables evaluated on the Jordan algebra $S(n,*)$ (of $n\times n$ symmetric elements of $UT_n(K)$ with respect to the transpose-like involution on $UT_n(K)$) is a vector space.

In this paper, we prove Conjecture \ref{LKUTn}, i.e., that the image of an arbitrary multilinear polynomial over an infinite field $K$ is $J^r$ for some $r$. To that we introduce the notion of \emph{commutator-degree} of a polynomial and classify the polynomials of a given commutator degree in terms of its coefficients. Finally, we prove that a multilinear polynomial has commutator degree $r$ if and only if its image on $UT_n(K)$ is $J^r$, i.e., our result not only describes the possible images of a given multilinear polynomial, but also for each possible image, it describes the set of multilinear polynomials with that image.

\section{Preliminaries}

In this paper, unles otherwise stated, $K$ will denote an infinite field of arbitrary characteristic and all algebras are associative.

If $X$ is a set, we denote by $\KX$ the free associative algebra freely generated by $X$. The elements of $\KX$ are polynomials in the noncommutative variables of $X$. 

Given an algebra $A$ over $K$, and a polynomial $f=f(x_1, \dots, x_m)\in \KX$, we denote also by $f$ the map \[\begin{array}{cccc}
     f: & A^m & \longrightarrow & A \\
        & (a_1,\dots, a_m) & \mapsto & f(a_1, \dots, a_m) 
\end{array}\] 
The image of such map is called \emph{the image of $f$ on $A$} and will be denoted by $f(A)$.

Image of polynomials have been studied long time before in the language of polynomial identities and central polynomials. If $A$ is an algebra over $K$, a polynomial $f\in \KX$ is a polynomial identity (PI) for $A$ if $f$ vanishes under substitutions of variables by elements of $A$. If any evaluation of $f$ by elements of $A$ yields central elements of $A$, we say that $f$ is a central polynomial. In the language of images of polynomials, a polynomial identity of $A$ is a polynomial whose image is $\{0\}$ and a central polynomial of $A$ is a polynomial whose image is contained in the center of $A$. An algebra $A$ satisfying a nonzero polynomial identity is called a PI-algebra.

For a given PI-algebra $A$, the set $Id(A) = \{f\in \KX \,|\, f \text{ is an identity of A}\}$ is an ideal of $\KX$ which is invariant under any endomorphism of the algebra $\KX$, i.e., it is invariant under substitution of variables by arbitrary elements of $\KX$. An ideal with this property is called a T-ideal of $\KX$. We refer the reader to \cite{Drensky, G-Z} for the basic theory of PI-algebras. 

Given a set of polynomials $\mathcal F \subseteq \KX$, we say the ideal $I$ is the T-ideal generated by $\mathcal F$, if $I$ is the smallest T-ideal containing $\mathcal F$. In this case we say $\mathcal F$ is a basis for $I$ or that the elements of $I$ follow from the elements of $\mathcal F$ and we denote $I=\langle \mathcal F \rangle ^T$.

% %%(Thiago): Mantém ou exclui o parágrafo a seguir?
% By a remarkable result of  Kemer \cite{Kemer}, it is well-known that if the characteristic of $K$ is zero, any T-ideal of $\KX$ is finitely generated as a T-ideal. This means that if $A$ is an algebra over $K$, all identities of $A$ follow from a finite set of its identities. Although the existence of a finite basis is known, it is a very difficult problem to find a finite basis for the T-ideal of a given algebra. For instance, for the algebra of $n\times n$ matrices over $K$,  $M_n(K)$, this is known only for $n\leq2$. 

Throughout the paper $UT_n(K)$ will denote the set of upper triangular matrices over a given field $K$. If $1\leq i \leq j \leq n$, we denote by $E_{i,j}$ the $n\times n$ matrix with $1$ in the entry $(i, j)$, and $0$ elsewhere. These will be called matrix units. In particular, the $E_{i,j}$ with $1\leq i \leq j \leq n$ is a basis for the vector space $UT_n(K)$. 

The set of strictly upper triangular matrices is the Jacobson radical of $UT_n(K)$. It is an ideal of $UT_n(K)$ and will be denoted by $J$. If $r\in\{1, \dots, n\}$, the power $J^r$ is the linear span of the set of all $E_{i,j}$ with $j-i \geq r$. For convenience, we denote $J^0=UT_n(K)$.

A finite basis of identities of $UT_n(K)$ were first described by Yu. N. Maltsev in the characteristic zero case. This result was proved by many other authors for arbitrary infinite fields. The classical proof can be found in \cite[Chapter 5]{Drensky}. An interesting proof based on a theorem J. Lewin can be found in \cite[Chapter 1]{G-Z}. 

\begin{theorem}\label{IDsUTn}
If $K$ is an infinite field, the polynomial identities of $UT_n(K)$ follow from the identity
		\[[x_1, x_2]\cdots [x_{2n-1}, x_{2n}] = 0\]
\end{theorem}
Here $[x,y]:=xy-yx$ denotes the commutator of $x$ and $y$.

A polynomial $f(x_1, \dots, x_m)\in \KX$ is called multihomogneous of multidegree $(n_1, \dots, n_m)$ if in each monomial of $f$, the variable $x_i$ has degree $n_i$. The polynomial $f$ is called multilinear if it is multihomogeneous of degree $(1, \dots, 1)$. It is clear that if the polynomial $f$ is multilinear, the corresponding map $f$ is a multilinear map between the vector spaces $A^m$ and $A$. Also, if $f(x_1, \dots, x_m)$ is multilinear, one can write 
\[f(x_1, \dots, x_m)=\sum_{\sigma\in S_m}\alpha_{\sigma}x_{\sigma(1)}\cdots x_{\sigma(m)},\]
for some $\alpha_{\sigma}\in K$, where $S_m$ denotes the symmetric group on $\{1, \dots, m\}$.

It is well known that if $K$ is an infinite field, any T-ideal is generated by multihomogeneous polynomials, and if $K$ has characteristic zero, by multilinear polynomials. This follows from the following fact which can be found in \cite{BresarKlep}.

\begin{lemma}
    Let $V$ be a vector space over $K$ and $U$ be a subspace of $V$. Let $c_0, c_1, \dots, c_n \in V$ such that $\sum_{i=0}^n\lambda^ic_i\in U$ for at least $n+1$ pairwise distinct scalars $\lambda$.    Then $c_i\in U$ for each $i\in\{0, \dots, n\}$.
\end{lemma}

In this paper we also need the notion of Lie ideal of an associative algebra $A$. If $I\subseteq A$ is a vector subspace of $A$, we say that $I$ is a Lie ideal of $A$ if for any $x\in A$ and $y\in I$, we have $[x,y]\in A$. As examples, the sets $\{0\}$, $K$, $sl_n(K)$ and $M_n(K)$ are Lie ideals of $M_n(K)$. Actually, these were shown to be the only Lie ideals of $M_n(K)$, if the characteristic of $K$ is not 2 or $n\neq 2$ (see \cite{Herstein}). 

The Lie ideal structures appear when studying images of polynomials. The following is a particular case of \cite[Theorem 2.3]{BresarKlep}.

\begin{theorem}\label{LieIdeal}
    Let $K$ be an infinite field, and let $A$ be an algebra over $K$. Then for every polynomial $f = f(x_1, \dots, x_m)\in \KX$, the linear span of $f(A)$ is a Lie ideal of A.
\end{theorem}

A Lie ideal $I$ of $\KX$ will be called a T-Lie ideal if  $I$ is a Lie ideal of $\KX$ and if it is closed under endomorphisms of $\KX$. We say that a T-Lie ideal $I$ is generated as a T-Lie ideal by some subset $\mathcal F\subseteq \KX$, if $I$ is the smallest T-Lie ideal containing $\mathcal F$. In this case we denote it by $I=\langle \mathcal F \rangle^{TL}$.

\section{The Commutator-Degree of a Polynomial}

In this section we introduce the notion of \emph{commutator-degree} of a polynomial $f\in \KX$, and give a characterization of multilinear polynomials of commutator-degree $r$ in terms of its coefficients. 

In order to define such notion, one first needs to observe that we have a strictly descending chain of T-ideals of $\KX$,
\[\KX \supsetneqq \langle [x_1,x_2] \rangle^T\supsetneqq \langle [x_1,x_2][x_3,x_4]\rangle^T\supsetneqq\langle [x_1,x_2][x_3,x_4][x_{5},x_{6}]\rangle^T\supsetneqq \cdots\]

\begin{definition}
    Let $f\in \KX$. We say that $f$ has \emph{commutator-degree} $r$ if $f\in \langle [x_1,x_2][x_3,x_4]\cdots [x_{2r-1},x_{2r}] \rangle^T$ and $f\not \in \langle [x_1,x_2][x_3,x_4]\cdots [x_{2r+1},x_{2r+2}] \rangle^T$.
\end{definition}
For convenience we say that $f$ has commutator-degree 0 if $f\in \KX$ but $f\not\in \langle [x_1, x_2] \rangle^T$.
%It is well known that the ideal of identities of the algebra $UT_n(K)$ is generated as a T-ideal by the product of $n$ commutators. That is, \[Id(UT_n)=\langle [x_1,x_2][x_3,x_4]\cdots [x_{2n-1},x_{2n}] \rangle^T.\] 
An immediate consequence of Theorem \ref{IDsUTn} is the following proposition.

\begin{proposition}\label{identities}
    Let $f\in \KX$. Then $f$ has commutator-degree $r$ if and only if $f\in Id(UT_r)$ and $f\not\in Id(UT_{r+1})$.
\end{proposition}

The definition of commutator degree was inspired in the following result (see \cite[Lemma 3]{deMello}). In \cite{deMello} such lemma is presented without proof, so we present its proof here for the sake of completeness.

\begin{lemma}\label{8}
	Let $A$ be a unitary algebra over $K$ and let \[f(x_1,\dots,x_m)=\sum_{\sigma\in S_n} \alpha_{\sigma} x_{\sigma(1)} \cdots x_{\sigma(m)}.\] be a multilinear polynomial in $\KX$.  
	\begin{enumerate}
		\item If $\sum_{\sigma \in S_n}\alpha_{\sigma}\neq 0$, then $f(A) = A$.
		
		\item $f\in \langle [x_1,x_2]\rangle ^T$ if and only if $\sum_{\sigma \in S_n} \alpha_{\sigma}=0$.
	\end{enumerate}
\end{lemma}

\begin{proof}
    To prove (1) one just need to substitute $m-1$ variables by $1\in A$ and the last variable by  $(\sum_{\sigma\in S_m}\alpha_\sigma)^{-1} a$ where $a\in A$ is an arbitrary element of $A$. This implies $f(A)=A$.
    
    To prove (2) we first assume that $f\in \langle [x_1,x_2]\rangle ^T$. Then using the identity $[xy,z]=x[y,z]+[x,z]y$, it is enough to verify it for a polynomial of the form $x_1\cdots x_{i-1} [x_i,x_{i+1}]x_{i+2}\cdots x_m$, which is obvious.
    
    Finally, we assume $\sum_{\sigma\in S_m}\alpha_\sigma=0$. Then for any $\tau\in S_m$, $\alpha_{\tau}=-\sum_{\sigma\neq \tau}\alpha_\sigma$ and $f$ can be written as
    \[f(x_1, \dots, x_m)=\sum_{\substack{\sigma\in S_m \\ \sigma\neq \tau}}\alpha_{\sigma}(x_{\sigma(1)}\cdots x_{\sigma(m)}-x_{\tau(1)} \cdots x_{\tau(m)}).\]
    By the above, it is enough to prove that for each $\sigma\neq \tau$, $x_{\sigma(1)}\cdots x_{\sigma(m)}-x_{tau(1)} \cdots x_{\tau(m)}\in \langle [x_1, x_2]\rangle^T$.
    We do it by induction on $m$. If $m=2$, we have $x_1x_2-x_2x_1\in \langle[x_1,x_2]\rangle^T$. Now assume the above is true for $m-1$. Up to relabeling variables, we may assume $\tau=id$ (the identical permutation). For a given $\sigma\in S_m$, let $j=\sigma(1)$. Then we can write 
    \begin{align*}
        & x_{\sigma(1)}\cdots x_{\sigma(m)}-x_1 \cdots x_m\\
         &= x_{\sigma(1)}\cdots x_{\sigma(m)}-x_jx_1\cdots x_{j-1}x_{j+1}\cdots x_m + x_jx_1\cdots x_{j-1}x_{j+1}\cdots x_m  - x_1 \cdots x_m\\
         &= x_j(x_{\sigma(2)}\cdots x_{\sigma(m)}-x_1\cdots x_{j-1}x_{j+1}\cdots x_m) + [x_j,x_1\cdots x_{j-1}]x_{j+1}\cdots x_m.
    \end{align*}
    The induction hypothesis implies that \[x_j(x_{\sigma(2)}\cdots x_{\sigma(m)}-x_1\cdots x_{j-1}x_{j+1}\cdots x_m)\in \langle[x_1, x_2]\rangle^T,\]
    and since $[x_j,x_1\cdots x_{j-1}]x_{j+1}\cdots x_m\in \langle[x_1,x_2]\rangle^T$ trivially, we obtain \[x_{\sigma(1)}\cdots x_{\sigma(m)}-x_1 \cdots x_m\in \langle[x_1, x_2]\rangle^T\] and the lemma is proved.
\end{proof}

The above lemma is a tool to determine if a polynomial lies in the T-ideal generated by $[x_1, x_2]$ only by knowing some relations satisfied by the coefficients of $f$, i.e, it characterizes the multilinear polynomials with commutator degree 0, and shows that polynomials of commutator degree 0 have image equal to $UT_n(K)$. The aim of this section is to generalize such characterization for multilinear polynomials of commutator-degree greater than 0 in order to apply it to the solution of the Lvov-Kaplansky conjecture for upper triangular matrices.

The following can be found in \cite{FagundesdeMello}.

\begin{lemma}\label{units}
    If $f$ is a multilinear polynomial, any evaluation of $f$ on upper triangular matrix units is either zero or some nonzero multiple of an upper triangular matrix unit.
\end{lemma}

Let now $f=f(x_1, \dots, x_m) \in \KX$ be a multilinear  polynomial and write it as \[f(x_1,\dots,x_m)=\sum_{\sigma\in S_n} \alpha_\sigma x_{\sigma(1)}\cdots x_{\sigma(m)}.\]

We now define some constants depending on the coefficients of $f$ and some subsets of $\{1,\dots,m\}$.

If $k\leq m$, let $h_1,\cdots,h_k\in \{1,2,\cdots,m\}$,  such that $h_1+\cdots+h_k\leq m$, let $t=(t_1,\cdots,t_k)$ be a tuple of elements in $\{1,2,\cdots,m\}$, with $t_1<t_2<\cdots <t_k$ and let $T=(T_1,T_2,\cdots,T_k)$ such that for all $j$, $T_j\subset \{1,2,\cdots,m\}$, $T_i\cap T_j=\emptyset$, if $i\neq j$ and $|T_i|=h_i-1,$ for $i\in \{1, 2, \dots, k\}.$ 

If $k\geq 1$, let us denote by $S(k,T,t)$ the subset of $S_m$ consisting of all permutations $\sigma$ satisfying:
\begin{itemize}  
    \item $\sigma(\{1,2,\cdots,h_1-1\})=T_1$
    \item $\sigma(h_1)=t_1$
    \item for $i\in \{2,\dots, k\}$, $\sigma(\{h_1+\cdots +h_{i-1}+1, \dots, h_1+\cdots + h_i-1\})=T_i$
    \item for $i\in \{2,\dots, k\}$, $\sigma(h_1+\cdots +h_i)=t_i$
\end{itemize}
If $k=0$, we consider $S(0)=S_m$
where $|T_i|=h_i-1$.

And now consider the following sums of coefficients of $f$:
\[\beta^{(0)}=\sum_{\sigma \in S_m}\alpha_{\sigma} \ \ \ \text{ and } \ \ \  \beta^{(k,T,t)}=\sum_{\sigma \in S(k,T,t)}\alpha_{\sigma}\]

The next two lemmas show how the above constants can be used to determine whether the multilinear polynomial $f$ is an identity for upper triangular matrices of a given order or not. In particular, the next is a generalization of Lemma \ref{8}.

\begin{lemma}\label{is}
    Let $f(x_1,\dots,x_m)=\sum_{\sigma\in S_n} \alpha_\sigma x_{\sigma(1)}\cdots x_{\sigma(m)}$ be a multilinear polynomial in $\KX$. If $\beta^{(k, T, t)}=0$, for all $k< n$, and for any $T$, and $t$, then $f\in Id(UT_n)$.
\end{lemma}

\begin{proof}
    Since $f$ is multilinear, it is enough to verify that $f$ vanishes under substitution of variables by matrix units. So let $r_1, \dots, r_m$ be matrix units. Among them, assume that $E_{i_1,j_1}, \dots, E_{i_k,j_k}$ are the strictly upper triangular ones. If $k\geq n$, then $f$ vanishes, so we may assume $k<n$. Let $r=f(r_1, \dots, r_m)$. Up to a permutation of indexes, we may assume  
    \[i_1<j_1=i_2<j_2=\cdots = i_k<j_k\] otherwise $r=0$. For $r_l\not\in\{E_{i_1,j_1}, \dots, E_{i_k,j_k}\}$, we must have
    \[r_l\in\{E_{i_1,i_1},E_{i_2,i_2}, \dots, E_{i_k,i_k}, E_{j_k,j_k}\}\] otherwise, $r=0$.
    
    For each $\sigma\in S_m$, the product $r_{\sigma(1)}\cdots r_{\sigma(m)}$ is nonzero only if it coincides with
    \[(E_{i_1,i_1}\cdots E_{i_1,i_1})E_{i_1,i_2}(E_{i_2,i_2}\cdots E_{i_2,i_2})E_{i_2,i_3}\cdots E_{i_k,j_k}(E_{j_k,j_k}\cdots E_{j_k,j_k})=E_{i_1,j_k}.\]
    
    Thus let $t_1, \dots, t_k$ be such that $r_{t_l}=E_{i_l,j_l}$ and let \[T_l=\{i\,|\,r_i=E_{i_l,j_l}\}, \quad l=1, \dots, k.\]
    From the considerations above, we obtain 
    \begin{align*}
        f(r_1, \dots, r_m)  & =\sum_{\sigma\in S_m}\alpha_\sigma r_{\sigma(1)}\cdots r_{\sigma(m)}\\
                            & =\sum_{\sigma\in S(k,T,t)}\alpha_\sigma r_{\sigma(1)}\cdots r_{\sigma(m)}\\
                            & =\beta(k,T,t)E_{i_1,j_k}=0.
    \end{align*}
    So, in any case, $f(r_1, \dots, r_m)=0$.
    \end{proof}

\begin{lemma}\label{isnot}
    Let $f(x_1,\dots,x_m)=\sum_{\sigma\in S_n} \alpha_\sigma x_{\sigma(1)}\cdots x_{\sigma(m)}$ be a multilinear polynomial in $\KX$ and let $n>0$. If there exist $T$, and $t$  such that $\beta^{(n, T, t)}\neq 0$, then $f\not\in Id(UT_{n+1})$.
\end{lemma}

\begin{proof}
    Let us consider the following evaluation of $f$ on matrix units of $UT_{n+1}$:
    \begin{itemize}
        \item $x_{t_i}\mapsto E_{i,i+1}$, for $i\in\{1,\dots,n\}$.
        \item $x_i\mapsto E{i,i}$, if $i\in T_i$. 
        \item $x_i\mapsto E_{n+1,n+1}$ otherwise.
    \end{itemize}
    Under such evaluation, the only non-vanishing monomials are those indexed by $\sigma \in S(n,T,t)$, and for any such $\sigma$ the resulting evaluation is $E_{1,n+1}$. Under such evaluation $f$ will result in $(\sum_{\sigma\in S(n,T,t)}\alpha_\sigma)E_{1,n+1}=\beta^{(n,T,t)}E_{1,n+1}$, which is nonzero. Then $f\not \in Id(UT_{n+1})$.
\end{proof}

As an immediate consequence of Lemma \ref{is} and Lemma \ref{isnot}, we have the main result of this section, which is a characterization of multilinear polynomials of commutator-degree $r$ in terms of its coefficients.

% \begin{theorem}
%     Let $f(x_1,\dots,x_m)=\sum_{\sigma\in S_n} \alpha_\sigma x_{\sigma(1)}\cdots x_{\sigma(m)}$ be a multilinear polynomial in $\KX$. Then $f$ has commutator-degree $s\geq 1$ if and only if \begin{itemize}
%         \item for all $k< s$, and for any $T$, and $t$. we have $\beta^{(k, T, t)}=0$
%         \item there exist $T$, and $t$  such that $\beta^{(s, T, t)}\neq 0$
%     \end{itemize}
% \end{theorem}

\begin{theorem}\label{equivalent}%[Enunciado alternativo do teorema acima - decidir qual é melhor deixar]
    Let $f(x_1,\dots,x_m)=\sum_{\sigma\in S_n} \alpha_\sigma x_{\sigma(1)}\cdots x_{\sigma(m)}$ be a multilinear polynomial in $\KX$. Then the following assertions are equivalent.
    \begin{enumerate}
        \item The polynomial $f$ has commutator-degree $r\geq 1$
        \item $f\in Id(UT_r)\setminus Id(UT_{r+1})$
        \item For all $k< r$, and for any $T$, and $t$, we have $\beta^{(k, T, t)}=0$ and there exist $T$, and $t$  such that $\beta^{(r, T, t)}\neq 0$
    \end{enumerate}
\end{theorem}

To illustrate concept of commutator-degree, we compute the commutator degree of the standard polynomials.

\begin{example}
    Let $St_m(x_1,\dots, x_m)=\sum_{\sigma\in S_m}(-1)^{\sigma} x_{\sigma(1)}\cdots x_{\sigma(m)}$ be the standard polynomial of degree $m$. Then for any $r$, $St_{2r}$ and $St_{2r+1}$ have commutator-degree $r$.
\end{example}

\begin{proof}
    From the well-known Amistsur-Levitzky theorem  $St_{2m}$ is a polynomial identity for $M_m(K)$, and hence for $UT_m$. Since $St_{2m+1}$ is a consequence of $St_{2m}$, it is also an identity for $UT_m$. The staircase lemma shows that these two polynomials are not identities for $UT_{m+1}$. 
\end{proof}

Actually, the above is a particular case of a more general result (see \cite[Lemma 9.1.2]{G-Z})  that the block-triangular matrix algebra $UT(d_1, \dots, d_s)$ satisfies the standard identity $St_k$ if and only if the sizes of the blocks $d_1, \dots, d_s$ satisfy $2(d_1+\cdots+d_s)\leq k$.

%%(Thiago): verificar se é melhor deixar o que segue no final dessa seção mesmo, ou se é melhor colocar no final do texto.

The study of the standard polynomials and a characterization of its multilinear consequences in terms of its coefficients may be an interesting tool to approach the Lvov-Kaplansky conjecture for $M_n(K)$, since we know from the celebrated theorem of Amistsur-Levitzki that for any $n$, $M_n(K)$ satisfies $St_{2n}$. Although we have computed the commutator-degree of the standard polynomials, this is a too coarse measure to verify if some polynomials is a consequence of some standard identity or not. One can argue in a similar way as it was done to commutator-degree to define a more refined tool. Let us call it the \emph{St-degree} of a polynomial: consider the descending chain of T-ideals
\[\KX \supsetneqq \langle St_{2}\rangle^T \supsetneqq \langle St_{3}\rangle^T \supsetneqq \langle St_{4}\rangle^T\supsetneqq  \cdots \]

A polynomial $f$ has St-degree $k$ if $f\in \langle St_{k}\rangle^T$ and $f \not \in \langle St_{k+1}\rangle^T$.

\begin{problem}
    Characterize the St-degree of multilinear polynomials in terms of its coefficients.
\end{problem}

\section{Images of multilinear polynomials on $UT_n(K)$}

In this section we present the main result of the paper, which gives a complete solution of the Lvov-Kaplansky conjecture for the algebra of $n\times n$ upper triangular matrices, i.e., a solution for Conjecture \ref{LKUTn}. Moreover, for each possible subspace which is the image of some multilinear polynomial, we give a complete description of the set of all multilinear polynomials having such an image. 

Before that, we recall a well-known result about commutative polynomials over infinite fields.

\begin{lemma}\label{nonroot}
    Let $K$ be an infinite field and let  $f(x_1,\dots,x_m)\in K[x_1,\dots, x_m]$ be a nonzero polynomial. Then, there exists $(a_1, \dots, a_m)\in K^m$ such that $f(a_1, \dots, a_m)\neq 0$.
\end{lemma}

In the proof of the main result, we will need the following consequence of the above:

\begin{lemma}
    Let $K$ be an infinite field and let $f_1(x_1, \dots, x_m), \dots, f_k(x_1, \dots, x_m)$ be nonzero polynomials in $K[x_1, \dots, x_m]$. Then, there exist $a_1, \dots, a_m\in K$ such that $f_i(a_1, \dots, a_m)\neq 0$, for $i=1, \dots, k$.
\end{lemma}

\begin{proof}
    Let $f(x_1, \dots, x_m)=\prod_{i=1}^{k}f_i(x_1, \dots, x_m)$. Since each $f_i$ is a nonzero polynomial, $f$ is also a nonzero polynomial, and the above lemma guarantees the existence of elements $a_1, \dots, a_m\in K$ such that $f(a_1, \dots, a_m)\neq 0$. The lemma is proved once one observes that 
    \[\prod_{i=1}^{k}f_i(a_1, \dots, a_m) =  f(a_1, \dots, a_m)\neq 0.\]
\end{proof}

\begin{theorem}\label{main}
    Let $f\in \KX$ be a multilinear polynomial. Then the image of $f$ on $UT_n(K)$ is $J^r$ if and only if $f$ has commutator-degree $r$.   
\end{theorem}

\begin{proof}%[Proof of Theorem \ref{main}]
        
    Let $f(x_1,\dots,x_m)$ be a multilinear polynomial of commutator-degree $r$. We will show that the image of $f$ on $UT_n(K)$ is $J^r$. We may assume that $r<n$, otherwise, $f$ is an identity for $UT_n(K)$.
    
    Let $t=(t_1,\dots, t_r)$, and $T=(T_1, \dots, T_r)$ such that $\beta(r, T, t)\neq 0$. We define $T_{r+1}=\{1,\dots, m\}\setminus (T_1\cup \cdots \cup T_r \cup \{t_1, \dots, t_r\})$.
    
    Given $A%=\sum_{j\geq i+r}a_{i,j}E_{i,j}
    \in J^r$, we can write \[A=\sum_{q=r}^{n-1}\sum_{k=1}^{n-q} a_{k,k+q}E_{k,k+q}\] We will show that there exists a suitable evaluation of variables $x_1\mapsto c_1, \dots, x_m\mapsto c_m$ such that $f(c_1,\dots, c_m)=A$.
    
    Before that, we will consider some matrices with entries in the polynomial algebra $K[d_j^{(i)}, y_{j}^{(i)}, x_{k, l}]$:
    \begin{itemize}
        \item $D_i=\sum_{j=1}^n d_{j}^{(i)}E_{j,j}$, for $i\in \{1, \dots, r+1\}$.
        \item $N_i=\sum_{j=1}^{n-1}y_j^{(i)}E_{j,j+1}$, for $i\in \{1, \dots, r\}$.
        \item $X=\sum_{k<l}x_{k,l}E_{k,l}$.
    \end{itemize} 
    
    Now we define the matrices $c_i$ depending on variables $d_j^{(i)}$, $y_{j}^{(i)}$ and $x_{k, l}$, so that we later can apply substitutions of such variables by elements of $K$ yielding a solution to the equation $f(c_1, \dots, c_m)=A$.

    The evaluation is given as follows:
    
    \begin{itemize}
        \item $c_i=D_i$, if $i\in T_i$, for $i\in \{1,\dots, r+1\}$.
        \item $c_{t_i}=N_i$, if $i\in \{1, \dots, r-1\}$.
        \item $c_{t_r}=X$.
    \end{itemize}
    Observe that if $r=1$, the matrices $N_i$ do not appear.
    
    Under such evaluation, since diagonal matrices commute, we obtain $f(c_1, \dots, c_m)$ is equal to
    
    \[\sum_{j=1}^{n}\sum_{i,\sigma}\gamma_{i,\sigma}D^{[i^{(1)}]}N_{\sigma(1)}D^{[i^{(2)}]}\cdots D^{[i^{(j)}]} X D^{[i^{(j+1)}]} \cdots D^{[i^{(r)}]}N_{\sigma(r)}D^{[i^{(r+1)}]},\]
    where:
    \begin{itemize}
        \item $D^{[i^{(k)}]}=D_1^{i_1^{(k)}}D_2^{i_2^{(k)}}\cdots D_{r+1}^{i_{r+1}^{(k)}}$,
        \item $\displaystyle \sum_{i,\sigma} =  \sum_{\substack{\sigma \in S_r\\ \sigma(j)=r}}  \sum_{p=1}^{r+1}\sum_{i_p^{(1)}+\cdots +i_{p}^{(r+1)}=|T_p|}$
        \item Each $\gamma_{i,\sigma}\in K$ depends on the $(r+1)^2$-tuple $i=(i_1^{(1)}, \dots, i_{r+1}^{(r+1)})$ and on $\sigma\in S_r$.
        \item If $i_0$ is the tuple satisfying $i_{q}^{(s)}=0$, for $q\neq s$, and if $id$ is the identity permutation of $S_r$, then $\gamma_{i_0, {id}}=\beta(r,T,t)$ (which is nonzero by hypothesis). 
    \end{itemize}
    
    Let us write $D^{[i^{(k)}]}=\sum_{j=1}^n \delta_j^{(k)}E_{j,j}$. Of course, for each $j$ and $k$, we have 
    \[\delta_{j}^{(k)}=\left(d_j^{(1)}\right)^{i_1^{(k)}} \cdots \left(d_j^{(r+1)}\right)^{i_{r+1}^{(k)}}.\]
    
    Computations with the above matrices yield
    \[ D^{[i^{(1)}]}N_{\sigma(1)}D^{[i^{(2)}]}\cdots N_{\sigma(j-1)}D^{[i^{(j)}]}  
    =  \sum_{k=1}^{n-j+1} \xi_{k}E_{k,k+j-1} \]
    where \[\xi_{k}=\prod_{s=1}^{j}\left(\delta_{k+s-1}^{(s)}\right)\prod_{s=1}^{j-1}\left(y_{k+s-1}^{(\sigma(s))}\right)\] and 
    \[ D^{[i^{(j+1)}]}N_{\sigma(j+1)} \cdots D^{[i^{(r)}]}N_{\sigma(r)}D^{[i^{(r+1)}]} 
    =  \sum_{k=1}^{n-r+j} \eta_{k}E_{k,k+r-j} \]
    where \[\eta_{k}=\prod_{s=1}^{r-j+1}\left(\delta_{k+s-1}^{(j+s)}\right)\prod_{s=1}^{r-j}\left(y_{k+s-1}^{(\sigma(j+s))}\right).\] 
    
    By multiplying the two expressions above with $X$ in the correct order, we obtain
    \begin{align*}
      & D^{[i^{(1)}]}N_{\sigma(1)}D^{[i^{(2)}]}\cdots D^{[i^{(j)}]} X D^{[i^{(j+1)}]} \cdots D^{[i^{(r)}]}N_{\sigma(r)}D^{[i^{(r+1)}]} \\
    = & \sum_{q=r}^{n-1}\sum_{k=1}^{n-q} \Omega_{i,j,q,k}x_{k+j-1,k+j-1+q}E_{k,k+q} 
    \end{align*}
    where
    \[\Omega_{i,j,q,k}=\left(\prod_{s=1}^{j}\delta_{k+s-1}^{(s)}\right)\left(\prod_{s=1}^{j-1} y_{k+s-1}^{(\sigma(s))}\right) 
    \left(\prod_{s=1}^{r-j+1}\delta_{k+j+q-r+s-1}^{(j+s)}\right)\left(\prod_{s=1}^{r-j} y_{k+j+q-r+s-1}^{(\sigma(j+s))}\right)\]
    
    % \begin{align*}
    % & D^{[i^{(1)}]}N_{\sigma(1)}D^{[i^{(2)}]}\cdots D^{[i^{(j)}]} X D^{[i^{(j+1)}]} \cdots D^{[i^{(r)}]}N_{\sigma(r)}D^{[i^{(r+1)}]} \\
    % = & \sum_{i=1}^{n-j-1}\sum_{l=i+j}^{n-r+j} \left(\prod_{k=0}^{j-2}\delta_{i+k}^{(k+1)}y_{k,i+k-1}\right)\delta_{i+j-1}^{(j)}\left(\prod_{k=0}^{r-j-1}\delta_{l+k}^{(j+k+1)}y_{j+k+1,l+k}\right)\delta_{l+r}^{(r+1)}x_{i+j-1,l} E_{i,l+r-j} 
    % \end{align*}
    
    If we write \[\Delta_{j,q,k}=\sum_{\substack{\sigma\in S_r \\ \sigma(j)=r}}\sum_{i}\gamma_{i,\sigma}\Omega_{j,q,k},\] we obtain
\begin{align*}
    f(c_1, \dots, c_m) & = \sum_{q=r}^{n-1}\sum_{k=1}^{n-q} \sum_{j=1}^{r} \Delta_{j,q,k}x_{k+j-1,k+j-1+q}E_{k,k+q} 
\end{align*}
    
    To obtain a solution for $f(c_1, \dots, c_m)=A$, we need to show that the following system of equations in variables $x_{k,l}$, $1\leq k<l\leq n$, has a solution.
    \[\sum_{j=1}^{r}\Delta_{j,q,k}x_{k+j-1,k+j-1+q}=a_{k, k+q},\quad q=r, \dots, n-1, k=1, \dots, n-q.\]
    
    The above equations can be written as
    \[\Delta_{r,q,k}x_{k+r-1,k+r-1+q} + \sum_{j=1}^{r-1}\Delta_{j,q,k}x_{k+j-1,k+j-1+q}=a_{k, k+q}.\]
   
    Let us assume that there exists an evaluation of variables $d_j^{(i)}$, $y_{j}^{(i)}$ by elements of $K$ such that under this evaluation  we have $\Delta_{r, q, t}\neq 0$, for each $q, k$. Then a solution of the above system of equations is obtained by considering $x_{k,l}=0,$ for $1\leq k\leq r-1$, $k+1\leq l \leq n$ and obtaining $x_{k,l}$ for all $k$ and $l$ recursively from the following formula:
    \[x_{k+r-1,k+r-1+q} = \frac{a_{k, k+q} - \sum_{j=1}^{r-1}\Delta_{j,q,t}x_{k+j-1,k+j-1+q}}{\Delta_{r,q,k}}.\]
    
    The existence of such an evaluation of variables $d_j^{(i)}$, $y_{j}^{(i)}$ by elements of $K$ such that for each $q, k$, $\Delta_{r, q, t}\neq 0$, is a consequence of Lemma \ref{nonroot}, once we prove that for each $q$ and $k$, $\Delta_{r,q,k}$ is a nonzero polynomial.
    
    To prove that, let us recall that
    \begin{align*}
        \Delta_{r,q,k} = & \sum_{\sigma\in S_{r-1}}\sum_{i}\gamma_{i,\sigma}\left(\prod_{s=1}^{r}\delta_{k+s-1}^{(s)}\right)\left(\prod_{s=1}^{r-1}y_{k+s-1}^{(\sigma(s))}\right)\delta_{k+q}^{(r+1)}\\
        = & \sum_{\sigma\in S_{r-1}}\sum_{i}\gamma_{i,\sigma}\left(\prod_{s=1}^{r}\prod_{l=1}^{r+1} \left(d_{k+s-1}^{(l)}\right)^{i_l^{(s)}}\right)\left(\prod_{s=1}^{r-1}y_{k+s-1}^{(\sigma(s))}\right)\prod_{l=1}^{r+1} \left(d_{k+q}^{(l)}\right)^{i_l^{(r+1)}}.
    \end{align*}
    
    Since \[\displaystyle \sum_{i} = \sum_{p=1}^{r+1}\sum_{i_p^{(1)}+\cdots +i_{p}^{(r+1)}=|T_p|},\] we observe that the multi-homogeneous degree of each monomial indexed by the $(r+1)^2$-tuple $i$ and by the permutation $\sigma$ is uniquely determined by $i$ and $\sigma$. Since $\gamma_{i_0, id}=\beta(r,T,t)\neq0$, the polynomial $\Delta_{r,q,k} $ is nonzero for each $q$ and $k$, and this finishes the proof of the theorem.
    \end{proof}

\section{Final Considerations}
We start the final considerations with a remark on the main theorem.

\begin{remark}
    Observe that for each possible image of a multilinear polynomial on $UT_n(K)$ (i.e., $J^r$), Theorem \ref{main} also describes the set of all multilinear polynomials with such image.
\end{remark}

Now let us consider images of not necessarily multilinear polynomials. It was shown in \cite{WangHomogeneous} that if $f$ is not a multilinear polynomial, its image may not be a vector subspace of $UT_n(K)$. Nevertheless, for each $r$, we may study the set all polynomials (not necessarily multilinear) with image contained in $J^r$. We have:

\begin{proposition}
    Assume $K$ is a field of characteristic zero. For each $r$, the set \[I^r(UT_n)=\{f\in \KX\,|\, f(UT_n)\subseteq J^r\}\] is a T-ideal of $\KX$. Moreover, for each $r$, $I^r(UT_n)= Id(UT_r)$. In particular, it is generated as a T-ideal by the polynomial $[x_1, x_2]\cdots [x_{2r-1}, x_{2r}]$. 
\end{proposition}

\begin{proof}
    The fact that $I^r(UT_n)$ is a T-ideal of $\KX$ follows directly from the fact that $J^r$ is an ideal of $UT_n(K)$ (and do not depend on the characteristic of $K$). 
    
    Since any T-ideal of $\KX$ is generated by its multilinear polynomials, the proposition follows directly from Theorem \ref{main} and Lemma \ref{equivalent}.
\end{proof}

The above result raises some questions when images of polynomials over matrix algebras (or even for algebras in general) are considered.

Assuming the Lvov-Kaplansky conjecture is true, the possible images of multilinear polynomials on $M_n(K)$ are $\{0\}$, $K$ (viewed as the set of scalar matrices), $sl_n(K)$ and $M_n(K)$.

The set of polynomials with image $\{0\}$ is the set of polynomial identities of $M_n(K)$, which is a T-ideal of $\KX$. The set of polynomials with image contained in $K$ is the set of central polynomials, denoted by $C(M_n(K))$, which is a T-space of $\KX$ (actually, a more suitable name should be T-subalgebra). It is a subalgebra of $\KX$ which is invariant under endomorphisms of the algebra $\KX$. Such kind of subsets of $\KX$ have been extensively studied, but in light of the Lvov-Kaplansky conjecture, one can also wonder about the set of polynomials with image contained in $sl_n(K)$. These were implicitly studied in the paper \cite{BresarKlep}. In order to make it explicit, we define the following set, which we call the set of \emph{traceless polynomials} and denote by $TL(M_n(K))$.
\[TL(M_n(K))=\{f(x_1, \dots, x_m)\in \KX\,|\, tr(f(a_1, \dots, a_m)=0\}.\]

An obvious but interesting remark is the following.

\begin{remark}
    The intersection of traceless polynomials with the central polynomials of $M_n(K))$ is the ideal of identities of $M_n(K)$. With the above notation:
    \[Id(M_n(K))=C(M_n(K))\cap TL(M_n(K)).\]
\end{remark}

We notice that $TL(M_n(K))$ has an interesting structure.

\begin{proposition}
    Let $TL(M_n(K))$ be as above. Then $TL(M_n(K))$ is a T-Lie ideal of $\KX$. Moreover, the T-Lie ideal generated by $[x_1, x_2]$ is contained in $TL(M_n(K))$, i.e.
    \[\langle[x_1, x_2] \rangle^{TL}\subseteq TL(M_n(K)).\]
    
\end{proposition}

\begin{proof}
    It follows directly from the fact that  $tr([A,B])=0$ for any $A, B\in M_n(K)$.
\end{proof}

We now observe that this is a particular case of a more general result.

\begin{proposition}
    Let $A$ be an algebra over $K$ and let $V_f\subseteq A$ be the linear span of the image of some polynomial $f$ on $A$. Then the set 
    \[\{g\in \KX\,|\, g(A)\subseteq V_j\}\]
    is a T-Lie ideal of $\KX$.
\end{proposition}

\begin{proof}
    This is an easy consequence of Theorem \ref{LieIdeal}.
\end{proof}

Last, we present a description of the set of traceless polynomials of $M_n(K)$ module $Id(M_n(K))$. To that, we recall the result below \cite[Theorem 4.5]{BresarKlep}. We say that $f$ is cyclically equivalent to a polynomial $g$ if $f-g$ is a sum of commutators in $\KX$.

\begin{theorem}
    Assume the characteristic of $K$ is zero. Let %$A$ be a finite dimensional central simple algebra, let 
    $f\in \KX$, and let us write $\mathcal L= \text{span } f(M_n(K))$. Then exactly one of the following four possibilities holds:
    \begin{enumerate}
        \item $f$ is an identity of $M_n(K)$; in this case $\mathcal L = 0$;
        \item $f$ is a central polynomial of $M_n(K)$; in this case $\mathcal L=K$;
        \item $f$ is not an identity of $M_n(K)$, but is cyclically equivalent to an identity of $M_n(K)$; in this case $\mathcal L = sl_n(K)$;
        \item $f$ is not a central polynomial of $M_n(K)$ and is not cyclically equivalent to an identity of $M_n(K)$; in this case $\mathcal L = M_n(K)$.
    \end{enumerate}
\end{theorem}

\begin{theorem}
    Assume the characteristic of $K$ is zero. Then the T-Lie ideal $TL(M_n(K))$ is generated (module $Id(M_n(K))$) by $[x_1, x_2]$ as a TL-ideal. Or, equivalently, $TL(M_n(K))$ is the T-Lie ideal generated by $Id(M_n(K))\cup \{[x_1,x_2]\}$.
\end{theorem}

\begin{proof}
    This is an easy consequence of the above theorem. Indeed, if $f$ is not an identity nor cyclically equivalent to an identity of $M_n(K)$, then the linear span of $f(M_n(K))$ is $K$ or $M_n(K)$. In both cases we do not have $f(M_n(K))\subseteq sl_n(K)$. So $f$ is cyclically equivalent to an identity of $M_n(K)$. The theorem is proved.
\end{proof}

\begin{remark}
    Working module $Id(M_n(K))$ is equivalent to work in the algebra of generic $n\times n$ matrices, $\mathcal G_n$ (see \cite[Chapter 7]{Drensky} for nice exposition on the algebra of generic matrices). In this language, the last theorem says that if an element of $\mathcal G_n$ has trace zero, then it is a sum of commutators of some elements of $\mathcal G_n$, as one can see in \cite[Theorem 5.1]{BresarKlep}.
\end{remark}

\section*{Funding}
T. C. de Mello was supported by grants \#2018/15627-2, and \#2018/23690-6 São Paulo Research Foundation (FAPESP).

\end{document}